\documentclass[11pt]{article}
\usepackage[utf8]{inputenc}
\usepackage[english]{babel}	
\usepackage[a4paper]{geometry}
\usepackage{fullpage}
\usepackage{amsmath, amsthm, amssymb, amsfonts}
\usepackage{mathtools}
\usepackage{thmtools}
\usepackage{enumerate}
\usepackage[shortlabels]{enumitem}
\usepackage{theoremref}
\usepackage{hyperref}
\usepackage{xcolor}
\usepackage{float}
\usepackage{appendix}
\usepackage{apptools,bbm}

\newcommand{\cP}{\ensuremath{\mathcal P}}

\newcommand{\NN}{\ensuremath{\mathbbm N}}

\renewcommand{\Pr}{\ensuremath{\mathbb P}}

\newcommand{\eps}{\varepsilon}
\renewcommand{\phi}{\varphi}

\DeclareMathOperator*{\E}{\mathbb{E}}
\DeclareMathOperator*{\N}{\mathbb{N}}

\renewcommand{\le}{\leqslant}
\renewcommand{\ge}{\geqslant}

\definecolor{royalazure}{rgb}{0.0, 0.22, 0.66}

\newtheorem{theorem}{Theorem}
\newtheorem{lemma}[theorem]{Lemma}
\newtheorem*{lemma*}{Lemma}
\newtheorem{claim}[theorem]{Claim}

\newtheorem{conjecture}[theorem]{Conjecture}

\newtheorem*{problem*}{Problem}

\theoremstyle{definition}

\theoremstyle{remark}
\newtheorem{remark}[theorem]{Remark}

\numberwithin{theorem}{section}
\numberwithin{lemma}{section}
\numberwithin{definition}{section}
\numberwithin{proposition}{section}
\numberwithin{conjecture}{section}
\numberwithin{remark}{section}
\numberwithin{claim}{section}
\numberwithin{question}{section}

\newcommand{\sm}{\setminus}
\newcommand{\se}{\subseteq}

\title{An approximate version of Jackson's conjecture}

\author{
  Anita Liebenau \thanks{School of Mathematics and Statistics, UNSW Sydney, NSW 2052, Australia. Email:~\texttt{a.liebenau@unsw.edu.au.} Supported by the Australian research council (ARC), DE170100789 and DP180103684.}		 
  \and
  Yanitsa Pehova\thanks{Mathematics Institute, University of Warwick, Coventry CV4 7AL, UK. Email: \texttt{y.pehova@warwick.ac.uk.} 
  This work has received funding from the European Research Council (ERC) under the European Union’s Horizon 2020 research and innovation programme (grant agreement No 648509). This publication reflects only its authors' view; the European Research Council Executive Agency is not responsible for any use that may be made of the information it contains.
  }
}
\date{}

\begin{document}

\maketitle

\begin{abstract}
A diregular bipartite tournament is a balanced complete bipartite graph whose edges are oriented so that every vertex has the same in- and outdegree. 
In 1981 Jackson showed that a diregular bipartite tournament contains a Hamilton cycle, and conjectured that in fact the edge set of it can be partitioned into Hamilton cycles. We prove an approximate version of this conjecture: for every $\eps>0$ there exists $n_0$ such that every diregular bipartite tournament on $2n\ge n_0$ vertices contains a collection of $(1/2-\eps)n$ cycles of length at least 
$(2-\eps)n$. Increasing the degree by a small proportion allows us to prove the existence of many Hamilton cycles: for every $c>1/2$ and $\eps>0$ there exists $n_0$ such that every $cn$-regular bipartite digraph on $2n\geq n_0$ vertices contains $(1-\eps)cn$ edge-disjoint Hamilton cycles. 
\end{abstract}

\section{Introduction}\label{sec:intro}

Finding sufficient conditions for a graph to contain a Hamilton cycle, i.e.~a cycle that contains every vertex of $G$, is one of the classical problems in graph theory. Dirac's theorem~\cite{dirac} states that every graph on $n$ vertices with minimum degree at least $n/2$ contains a Hamilton cycle. Later, Ore~\cite{ore} showed that it is enough if every pair of non-adjacent vertices has the sum of their degrees totaling at least $n$. A natural extension to the existence of one Hamilton cycle is then the existence of many edge-disjoint Hamilton cycles, or even of a decomposition into Hamilton cycles, i.e.~a partition of the edges of a graph into Hamilton cycles. Clearly, if such a decomposition exists, say into $d$ Hamilton cycles, then the graph must be $2d$-regular. A construction by Walecki (see, e.g., \cite{alspach,HILTON1984125}) shows that the complete graph $K_{2d+1}$ admits such a decomposition for every $d\ge  1$. More generally, the complete $r$-partite graph $K(n;r)$ on $rn$ vertices admits a decomposition into Hamilton cycles whenever $(r-1)n$ is even; and into Hamilton cycles and a perfect matching if $(r-1)n$ is odd~\cite{hetyei, la1976}. Some further graph classes have been shown to admit Hamilton decompositions, we refer the reader to the survey article by Alspach, Bermond and Sotteau~\cite{Alspach1990}. 

Nash-Williams \cite{nash-williams71} extended Dirac's theorem by showing that every $n$-vertex graph with minimum degree at least $n/2$ contains at least $5n/224$ edge-disjoint Hamilton cycles, and conjectured that the minimum degree condition is sufficient to prove the existence of $\lfloor \frac{n+1}{4}\rfloor$ edge-disjoint Hamilton cycles. Babai (see \cite{nash-williams70}) provided a construction showing that this is false. However, Csaba, K\"uhn, Lo, Osthus and Treglown \cite{csaba} proved that regular graphs satisfying the above minimum degree condition can be decomposed into Hamilton cycles and at most one perfect matching. 

These problems naturally extend to the setting of {\em oriented graphs} that are obtained from simple graphs by endowing every edge with an orientation. 
We write $(u,v)$ for the (oriented) edge between the pair $\{u,v\}$ oriented from $u$ to $v$. A Hamilton cycle in an oriented graph $G$ is an ordering  $v_1,\ldots,v_n$ of the vertices of $G$ such that for all $1\le i \le n$ the edge $(v_i,v_{i+1})$ is present (where $v_{n+1} = v_1$). The {\em outdegree} of a vertex $v$ in an oriented graph $G$, denoted by $d_G^+(v)$, is the number of edges $(v,y)\in E(G)$, and the {\em indegree} of a vertex $v$ in an oriented graph $G$, denoted by $d_G^-(v)$, is the number of edges $(x,v)\in E(G)$. We suppress the subscript $G$ if the graph $G$ is clear from context. We set $\delta^{+}(G)=\min_{v\in V(G)} d^+(v)$, $\delta^{-}(G)=\min_{v\in V(G)} d^-(v)$, and 
$\delta^{0}(G)=\min\{\delta^{+}(G), \delta^{-}(G)\}$. We refer to the latter one as the {\em minimum semidegree} of $G$ (the \emph{maximum semidegree} $\Delta^0(G)$ is defined analogously). 

Keevash, K\"{u}hn and Osthus~\cite{keevash} show that for $n$ large enough, every oriented graph $G$ on $n$ vertices with minimum semidegree at least $\frac{3n-4}{8}$ contains a Hamilton cycle. A construction due to H\"{a}ggkvist~\cite{haggkvist_1993} shows that this is best possible. K\"{u}hn and Osthus~\cite{KUHN201362} prove that every $r$-regular oriented graph $G$ on $n$ vertices has a Hamilton cycle decomposition for every $r\geq cn$, where $c>3/8$ is a constant and $n$ is large enough. In particular, this establishes Kelly's conjecture which states that every regular tournament has a Hamilton cycle decomposition. The result in~\cite{KUHN201362} builds on earlier work by K\"{u}hn, Osthus and Treglown~\cite{kot2010} which includes a first approximate version of Kelly's conjecture. 

How many disjoint Hamilton cycles can one guarantee when the (oriented) graph is not regular? As the union of disjoint Hamilton cycles forms a regular spanning subgraph, the maximal $r$ for which $G$ contains an $r$-regular spanning subgraph is an upper bound for this quantity. 
Ferber, Long and Sudakov \cite{ferber2018} show that this upper bound is asymptotically correct for oriented graphs of large enough minimum semidegree. 

\begin{theorem}[Ferber, Long, Sudakov \cite{ferber2018}]\label{thm:fls}
Let $c>3/8$, $\eps>0$ and let $n$ be sufficiently large. Let $G$ be an oriented graph on $n$ vertices with $\delta^0(G)\geq cn$. Then $G$ contains $(1-\eps)r$ edge-disjoint Hamilton cycles, where $r$ is the maximum integer such that $G$ contains an $r$-regular spanning subgraph.
\end{theorem}

In this paper, we consider the corresponding degree conditions for regular bipartite oriented graphs. An obvious necessary condition for a bipartite (oriented) graph to contain a Hamilton cycle is that both parts of the bipartition have equal size, in which case the graph is called \emph{balanced}. Note that the minimum  
semidegree of a bipartite oriented graph $G$ can be at most $\lfloor v(G)/4\rfloor$, where $v(G)$ denotes the number of vertices of $G$. Graphs that attain this bound satisfy $4|v(G)$, and are necessarily balanced and $(v(G)/4)$-regular. Such graphs are called \emph{diregular bipartite tournaments}. Jackson~\cite{jackson} showed that  
diregular bipartite tournaments are Hamiltonian, and he conjectured the following. 

\begin{conjecture}[Jackson \cite{jackson}]\label{conj:jackson}
Every diregular bipartite tournament is decomposable into Hamilton cycles.
\end{conjecture}

In this paper we adjust the methods of \cite{ferber2018} to the bipartite setting and prove the following relaxation of Jackson's conjecture. A {\em directed graph} (or \emph{digraph}, for short) consists of a set of vertices $V$ and a set of ordered pairs of $V$, called directed edges (or just edges). That is, directed graphs may contain edges $(x,y)$ and $(y,x)$ for two vertices $x,y\in V$, but no loops and no multiple edges. The notions of Hamilton cycles, minimum semidegree, etc., introduced earlier for oriented graphs generalise in the natural way to directed graphs. 

\begin{theorem}\label{thm:main}
Let $c>1/2$, $\eps>0$, and let $n$ be sufficiently large. Then every $cn$-regular  bipartite digraph $G$ on $2n$ vertices contains at least $(1-\eps)cn$ edge-disjoint Hamilton cycles. 
\end{theorem} 

To the best of our knowledge no other intermediate results towards Conjecture \ref{conj:jackson} are known. Our result constitutes an approximate version of Conjecture~\ref{conj:jackson} in the following sense. 
Let $G$ be a diregular bipartite tournament and add $\eps n$ in-  and out-neighbours to every vertex (this can be realised by adding an edge-disjoint $(\eps n)$-regular spanning subgraph using Hall's theorem on the complement). We obtain an almost decomposition of the resulting graph into Hamilton cycles. 
Of course, none of these Hamilton cycles need to be cycles of the original bipartite tournament. For bipartite oriented graphs we prove the following. 
\begin{theorem}\label{thm:main2}
Let $c>1/4$, $\eps>0$, and let $n$ be sufficiently large. Then every $cn$-regular bipartite oriented graph $G$ on $2n$ vertices contains at least $(1-\eps)cn$ edge-disjoint cycles of length at least $2n-O(n/\log^2 n).$ 
\end{theorem}
In particular, we can almost decompose the edge set of every bipartite regular tournament into almost spanning cycles. 

We note that the constants $1/2$ and $1/4$ in Theorems~\ref{thm:main} and~\ref{thm:main2} are optimal for such statements. Indeed, a $d$-regular digraph may be disconnected if $d=n/2$, as may be a $d$-regular oriented graph if $d=n/4.$

\section{Preliminaries}\label{sec:prelim}
In this section we introduce notation and present lemmas that we later use in the proof of our main result.

All graphs and digraphs are finite and simple, that is, they do not contain loops or double edges (in the case of graphs) or double oriented edges (in the case of digraphs). Let $G$ be a graph or a digraph. We denote by $V(G)$ the vertex set of $G$ and by $E(G)$ the edge set of $G$. For subsets $X,Y\se V(G)$ we write $E_G(X,Y)$ for the set of edges $xy$ if $G$ is a graph, and the set of directed edges $(x,y)$ if $G$ is a digraph. Let $G[X]$ denote the graph or digraph induced on $X$. When $G$ is a graph let $N_G(X)$ denote the set of vertices $y$ such that $xy\in E(G)$ for some $x\in X$. When $G$ is digraph, we denote by $N^-_G(X)$ ($N^+_G(X)$) the set of vertices $y$ such that $(y,x)\in E(G)$  ($(x,y)\in E(G))$ for some $x\in X$. 
When $X=\{v\}$ we also write $E_G(v,Y)$ (and $E_G(Y,v)$ in the digraph case), $N_G(v)$, $N^-_G(v)$, and $N^+_G(v)$ for the above sets, where the latter three we call the {\em neighbourhood}, the {\em in-neighbourhood}, and the {\em out-neighbourhood of $v$}, respectively. 
The sizes of these sets are denoted by $v(G) = |V(G)|$, $e(G)=|E(G)|$, $e(X,Y) = |E(X,Y)|$, $d_G(v)= |N_G(X)|$, $d^-_G(v)= |N^-_G(v)|$, $d^+_G(v)= |N^+_G(v)|$. We also write $d_G(v,Y)$ for $e_G(v,Y)$ when $G$ is a graph, and $d^+_G(v,Y)=e_G(v,Y)$ and $d^-_G(v,Y)= e_G(Y,v)$ when $G$ is a digraph. Throughout the paper, expressions of the form $d^\pm(v)\ge d$ are used as short-hand for "$d^-(v)\ge d$ and $d^+(v)\ge d$", and all other uses of $\pm$ carry the analogous meaning.  
We omit the subscript $G$ when there is no danger of ambiguity. 

We say a graph or digraph $G$ {\em has bipartition $(V_1,V_2)$} if $V(G)=V_1\cup V_2$ where $V_1$ and $V_2$ are disjoint and all edges have one endpoint in $V_1$ and one in $V_2$. A digraph $G$ is a {\em balanced bipartite} digraph if it has a bipartition $(V_1,V_2)$ such that $|V_1|=|V_2|$. 

For a graph or digraph with bipartition $(V_1,V_2)$ and a subset $W\se V(G)$ we write $W^{V_1}$ and $W^{V_2}$ for $W\cap V_1$ and $W\cap V_2$, respectively. 

For real numbers $x,y,z$ we write $x=y\pm z$ if $x\in [y-z,y+z]$. For two functions $f(n)$ and $g(n)$ we write $f(n)\ll g(n)$ if $f(n)/g(n)\to 0$ as $n\to\infty$. We omit floor and ceiling signs for clarity of presentation.

We need the following standard concentration result for binomial random variables (see \cite[Theorem A.1.1]{alonspencer}).

\begin{lemma}[Chernoff's inequality]\label{lem:chernoff}
Let $X$ be a binomial random variable with parameters $(n,p)$, and let $\mu = np$. Then 
\begin{align*}
    \Pr(|X-\mu|\ge a)\le 2e^{-a^2/3\mu}.
\end{align*} 
\end{lemma}

\begin{remark}\label{rem:hypergeom} 
Let $X$ be a hypergeometric random variable with parameters 
$(N,K,n)$, that is, given an underlying set $V$ of size $N$ and a subset $S\se V$ of size $K$,  $X=|Y\cap S|$ where $Y$ is a subset of $V$ of size $n$ chosen uniformly at random. The same inequality as in Lemma~\ref{lem:chernoff} holds for $X$, where now $\mu=nK/N.$
For details see \cite[Section 21.5]{frieze_karonski_2015}).
\end{remark}

The following provides a sufficient minimum semidegree condition for a digraph to contain a Hamilton cycle. 
\begin{theorem}[Ghouila-Houri \cite{ghouilahouri}]\label{thm:ghouila-houri}
Every strongly connected digraph $G$ on $n$ vertices with $\delta^+(G)+\delta^-(G)\geq n$ contains a Hamilton cycle. In particular, if $\delta^0(G)\geq n/2$, then $G$ contains a Hamilton cycle.
\end{theorem}

Let $D_{n,n}$ denote the complete bipartite balanced digraph in which both vertex classes have size $n$ and every vertex has in- and outdegree $n$.
A result by Ng~\cite{NG1997279} implies that the edge set of $D_{n,n}$ can be decomposed into Hamilton cycles. We use this to prove the following. 
\begin{lemma}\label{lem:completebipartite}
There exists $n_0\in \N$ such that for all $n\geq n_0$ the complete bipartite digraph $D_{n,n}$ contains $n$ disjoint Hamilton paths starting in the same vertex class of the bipartition. 
Moreover, 
every vertex of $D_{n,n}$ is an endpoint of at most $2\sqrt{\log n}$ of these paths. 
\end{lemma}
\begin{proof}
Let $A$ and $B$ denote the vertex classes of $D_{n,n}$. 
It follows from Ng~\cite{NG1997279} that there is a decomposition of 
$D_{n,n}$ into $n$ Hamilton cycles, say $C_1,...,C_n$. 
For every $i\in [n]$ choose an edge $e_i=(a_i,b_i)$ of $C_i$ with $a_i\in A$ uniformly at random among all $n$ such edges, all choices being independent. Denote their union by $H$. We claim that with positive probability  $\Delta^0(H)$ 
is at most $2\sqrt{\log n}$. 

Fix a vertex $v\in A$. Then for each vertex $w\in B$, the edge $(v,w)$ is in $H$ with probability~ $1/n$. Moreover, the events $E_w=\{\mbox{the edge }(v,w)\mbox{ is in }H\}$ are independent since for any two distinct vertices $w,w'\in B$ the edges $(v,w)$ and $(v,w')$ are in different cycles of the decomposition. Therefore, the out-degree of $v$ in $H$ has a binomial distribution with parameters $n$ and $1/n$. 
Similarly, the in-degree of $w$ in $H$ has a binomial distribution with parameters $n$ and $1/n$ for every $w\in B$. Therefore, the probability that there exists $v\in A$ with $d^+_H(v)>2\sqrt{\log n}$ or $w\in B$ with $d^-_H(w)>2\sqrt{\log n}$ is at most $4ne^{-4\log n/3}=o(1)$, by Chernoff's inequality (Lemma \ref{lem:chernoff}) and the union bound.
It follows that with positive probability $H$ has maximum semidegree at most $2\sqrt{\log n}$. 
The claim follows by taking $\{C_i-e_i\}_{i\in[n]}$, as the collection of Hamilton paths. By the choice of $e_i$'s all these paths start in $B$. 
\end{proof}

Finally, we use the following from~\cite{ferber2018}. 
\begin{lemma}[Lemma 24 in \cite{ferber2018}]\label{lem:lem23}
Let $\eps>0$ and $m,r\in \NN$ with $m$ sufficiently large and $2m^{24/25}\leq r \leq (1-\eps)m/2$. Suppose that $G=(A\cup B,E)$ is a bipartite graph with $|A|=|B|=m$ and $r\leq \delta(G) \leq \Delta(G)\leq r+r^{2/3}$. Then $G$ contains a collection of $r-m^{24/25}$ edge-disjoint matchings, each of which has size at least $m-m^{7/8}$, and whose union has minimum degree at least $r-m^{24/25}-2m^{5/6}$.
\end{lemma}
\begin{remark}\label{rem:NearlyBalanced}
Note that practically the same assertion holds when $|A| = m=|B|+1$, up to an additive constant of 1 which we neglect due to the asymptotic nature of the statement. To see this, apply the lemma to the graph obtained by adding an auxiliary vertex $v$ to $B$ and $\delta(G)$ edges between $v$ and $A$. 
\end{remark}

%
%
\section{Proof of Theorems~\ref{thm:main} and~\ref{thm:main2}}

In this section we prove our main theorems. The two proofs are very similar and we treat them together for a large part of this section. We state lemmas along the way that we either prove in the appendix (Lemma~\ref{lem:partition}) or at the end of the section (Lemmas~\ref{lem:pathcover},~\ref{lem:pathstohamcycle} and~\ref{lem:pathstoLONGcycle}). We introduce some notation specific to the proof. 

A \emph{path cover of size $k$ of a directed graph $H$} is a set $\cP$ of $k$ directed paths in $H$ such that every vertex is contained in exactly one path of $\cP$. Every digraph $H$ contains a {\em trivial path cover} in which every path consists of exactly one vertex of $H$, whereas a Hamilton path, if existent, is a path cover of size one. 
We call two path covers $\cP_1$ and $\cP_2$ {\em edge-disjoint} if any two paths $P_1\in \cP_1$ and $P_2\in \cP_2$ are edge-disjoint. 
Given a set of path covers ${\bf P}$ 
of a digraph $H$, we denote by 
$G_{\bf{P}}$ the graph whose edge set is formed by taking the union of all sets $E(P)$, for all paths $P\in \cP$, for all path covers $\cP\in {\bf P}$. 

Let $c>\eps>0$ where we may assume for the proof that $\eps$ is sufficiently small. Let $n$ be a sufficiently large integer. Let $d= cn$ and assume that $G$ is a balanced $d$-regular bipartite digraph on $2n$ vertices. 

The next lemma asserts that we can split $G$ into roughly $(\log n)^3$ spanning subgraphs, each with good degree conditions into certain subsets. 
\begin{lemma}\label{lem:partition}
Let $c>\eps>0$ be constants, let $n$ be sufficiently large. 
Let $D$ be a $d$-regular bipartite digraph with bipartition $(A,B)$ such that $|A|=|B|=n$, where $d= cn$. 
Then for $K=\log n$  there are $K^3$ edge-disjoint spanning subdigraphs 
$H_1,...,H_{K^3}$ of $D$ with the following properties. 
\begin{enumerate}[label={(P\arabic*)}]
    \item \label{PPartition} 
    For each $1\leq i\leq K^3$ there is a partition 
    $V(G)=U_i\cup W_i$ with $|W_{i}^A|=|W_{i}^B|=n/K^2\pm 1$;
    \item \label{PdegreesInUis}  
    For some $r= (1\pm \eps)d/K^3$ and all $1\leq i\leq K^3$, 
    the induced subgraph $H_i[U_i]$ satisfies 
    \[\delta^0(H_i[U_i]),\Delta^0(H_i[U_i])= r\pm r^{3/5}; \]
    \item \label{PdegreeIntoWis} 
    For all $1\leq i\leq K^3$ and all $u\in U_i$ we have that 
    $d^\pm_{H_i}(u,W_i)\geq \eps c |W_i|/8K$;
    \item \label{PsemiDegInWis}  
    Each induced subgraph $H_i[W_i]$ has minimum semidegree at least $(c-\eps)|W_i|/2$.
\end{enumerate}
\end{lemma}
The proof of the lemma is a straight-forward adaptation of Lemma~27 in~\cite{ferber2018} to the bipartite setting. We include it in the appendix for completeness. 

We now claim that each $H_i[U_i]$ as given by the previous lemma has many edge-disjoint path covers. Precisely, we prove the following. 
\begin{lemma} \label{lem:pathcover}
There exists a positive integer $m_0\in \N$, such that for $m\geq m_0$ and  $m^{49/50}\leq r\leq m/3$ the following is true. 
Let $H$ be a balanced bipartite digraph on $2m$ vertices 
such that $d^{\pm}(v) =r\pm r^{3/5}$ for every vertex $v$ of $H$. 
Then $H$ contains a collection ${\bf P}$ 
of at least $r-m^{24/25}\log m$ edge-disjoint path covers, each of size at most $m/\log^4 m$. Moreover, $\delta^0(G_{\bf{P}})\geq r-m/(\log m)^{39/10}$.
\end{lemma}

In the proofs of Theorems~\ref{thm:main} and~\ref{thm:main2}, respectively, we will apply Lemma~\ref{lem:pathcover} to each $H=H_i[U_i]$. The strategy is then to connect the paths of each path cover in $H_i[U_i]$ to a Hamilton cycle (Theorem~\ref{thm:main}) or to a long cycle (Theorem~\ref{thm:main2}) using the vertices in $W_i$ in such a way that the cycles corresponding to distinct path covers are edge disjoint. We make this precise using the following two lemmas. A subset $S$ of the vertices of a bipartite digraph $F$ with bipartition $(A,B)$ is called {\em balanced} if $|S^A|=|S^B|$. 

\begin{lemma}\label{lem:pathstohamcycle}
Let $c'>1/2$, and let $a,n'$ be positive integers such that, $a\ll n'/\log n'$. Let $F$ be a balanced bipartite digraph on $2n'$ vertices such that $\delta^0(F)\geq c'n'$. Then, given a balanced set of distinct vertices $s_1,t_1,...,s_a,t_a\in V(F)$ with respect to a balanced bipartition of $F$, there exists a path cover $\cP=\{P_1,...,P_a\}$ of $F$ such that each path $P_i$ starts at $s_i$ and ends at $t_i$.
\end{lemma}

\begin{lemma}\label{lem:pathstoLONGcycle}
Let $c'>1/4$, and let $a,n'$ be positive integers such that, $a\ll n'/\log n'$. Let $F$ be a balanced bipartite oriented graph on $2n'$ vertices such that $\delta^0(F)\geq c'n'$. Then, given a set of distinct vertices $s_1,t_1,...,s_a,t_a\in V(F)$, there exists a collection of pairwise vertex disjoint paths $\{P_1,...,P_a\}$ of $F$ such that each path $P_i$ starts at $s_i$ and ends at $t_i$. 
\end{lemma}

We are ready to prove our main theorems. 

\begin{proof}[Proof of Theorem~\ref{thm:main}]
Let $c>1/2$, $\eps>0$ where we may assume for the proof that $\eps$ is sufficiently small. Let $n$ be a sufficiently large integer. Let $d= cn$ and assume that $G$ is a balanced $d$-regular bipartite digraph on $2n$ vertices. 
Let $K=\log n$ and let $H_1,\ldots, H_{K^3}$ be the subdigraphs given by Lemma~\ref{lem:partition} satisfying the properties \ref{PPartition}--\ref{PsemiDegInWis}. 

For each $i\in[K^3]$ we apply Lemma~\ref{lem:pathcover} with $m= |U_{i}^A|=|U_{i}^B|=n-n/K^2\pm 1$ and $r$ given by \ref{PdegreesInUis}. Note that $r=(1\pm \eps)d/K^3=\Theta(n/K^3)$ and $H_i[U_i]$ is balanced so that the assumptions of Lemma~\ref{lem:pathcover} are satisfied for $H=H_i[U_i]$. 
Therefore, for every $i\in[K^3]$, we obtain a collection ${\bf P}^{(i)}$ of at least $r'=r-n^{24/25}\log n$ edge-disjoint path covers of $H_i[U_i]$, each of size at most $a=n/\log ^4 n$, and such that 
\begin{equation}\label{eq:aux081}
\delta^0(G_{{\bf P}^{(i)}})\geq r-n/(\log n)^{39/10}.
\end{equation} 

Now fix $i\in [K^3]$ and let $\cP_1^{(i)},\ldots,\cP_{r'}^{(i)}$ be $r'$ path covers of ${\bf P}^{(i)}$ as above. We iteratively find $r'$ edge-disjoint Hamilton cycles $C_1^{(i)},\ldots,C_{r'}^{(i)}$ in $H_i$ such that $C_k^{(i)}[U_i]$ consists exactly of the edges in $\cP_k^{(i)}$, for all $1\le k \le r'$. In other words, the paths in $\cP_k^{(i)}$ are connected to a cycle $C_k^{(i)}$ via edges in $E(U_i,W_i)\cup E(W_i,U_i)\cup E(W_i)$.  
For $1\le k\le r'$ suppose that we have obtained such $k-1$ edge disjoint Hamilton cycles $C_1^{(i)},\ldots,C_{k-1}^{(i)}$. Let $F_k$ be the graph obtained from $H_i$ by removing the edges of those $k-1$ cycles.
Let $(x_1,y_1),\ldots,(x_{\ell},y_{\ell})$ be the pairs of start and end points of the paths in $\cP_k^{(i)}$, and note that $\ell \le n/\log^4 n$. 
We now greedily pick pairwise distinct vertices $s_1,t_1,\ldots,s_{\ell},t_{\ell}\in W_i$ such that 
\begin{equation}\label{eq:connectingEdges}
(y_1,s_1),( t_1,x_2),\ldots,(y_{\ell},s_{\ell}),(t_{\ell},x_1) \in E(F_k).
\end{equation}
We verify briefly that this is indeed possible. For a vertex $v\in \{x_1,y_1,\ldots,x_{\ell},y_{\ell}\}\se U_i$ we have that $d_{H_i}^\pm (v,W_i)\ge \eps |W_i|/16K$, by~\ref{PdegreeIntoWis} and since $c>1/2.$ An edge in $E(v,W_i)$ (or $E(W_i,v)$, respectively) is removed from $H_i$ only if $v$ is the endpoint (or startpoint, respectively) of a path in 
$\bigcup_{j=1}^{k-1} \cP_j^{(i)}$ (and in this case, at most one edge is removed from $H_i$). 
Since $\delta^0(G_{{\bf P}^{(i)}})\geq r-n/(\log n)^{39/10}\ge r'-n/(\log n)^{39/10}$ by~\eqref{eq:aux081}, it follows that every $v\in U_i$ is the start (or end) point of at most 
$n/(\log n)^{39/10}$ paths in $\bigcup_{j=1}^{r'} \cP_j^{(i)}$. 
Thus, 
$$d_{F_k}^+(v,W_i)\ge d_{H_i}^+ (v,W_i) - n/(\log n)^{39/10}>0$$
at each step, and we can indeed pick $s_1,t_1,\ldots, s_{\ell},t_{\ell}$ greedily in $W_i$ such that 
\eqref{eq:connectingEdges} holds.

We verify that $F_k[W_i]$, together with the set $\{s_1,t_1,s_2,t_2,\ldots,t_\ell\}$ satisfies the assumptions of Lemma~\ref{lem:pathstohamcycle}. Note that $n'= |W_i^A|=n/K^2 \pm 1$. Furthermore, the path cover $\cP_k^{(i)}$ has size at most $n/\log^4n$, hence $\ell\le n/\log^4n\ll n'/\log n'.$  Now, 
$\delta^{0}(F_k[W_i])\ge (c-\eps)n' - (k-1)$ by~\ref{PsemiDegInWis} and since the only edges incident to vertices in $W_i$ that were removed from $H_i$ are those belonging to the Hamilton cycles $C_1^{(i)},\ldots,C_{k-1}^{(i)}$. This implies that $\delta^{0}(F_k[W_i])\ge c'n'$ for some $c'>1/2$, since $c>1/2$, $\eps>0$ is small enough, and $k\ll n'$.
Finally, the set of vertices $s_1,t_1,s_2,t_2,\ldots,t_\ell$ is balanced because the set $x_1,y_1,...,x_\ell,y_\ell$ of endpoints of paths in $\cP$ is also balanced.

Therefore, by Lemma~\ref{lem:pathstohamcycle}, $F_k[W_i]$ contains a path cover $\cP=\{P_1,\ldots,P_{\ell}\}$ such that $P_j$ is an $s_j$-$t_j$-path for $1\le j\le \ell$. These paths, together with the paths in $\cP_k^{(i)}$ and the edges in~\eqref{eq:connectingEdges} form a Hamilton cycle $C_{k}^{(i)}$ in $F_k\se H_i$ that is edge-disjoint from $C_1^{(i)},\ldots,C_{k-1}^{(i)}$ and from the paths in $\cP_{k+1}^{(i)},\ldots,\cP_{r'}^{(i)}$. 

Thus, after $r'$ iterations, we obtain the desired edge-disjoint Hamilton cycles 
$C_1^{(i)},\ldots,C_{r'}^{(i)}$ of $H_i$. 
Treating all $K^3$ subgraphs $H_i$ in parallel (recall that they were edge-disjoint), we obtain $K^3 r'\ge (1-2\eps)d$ edge-disjoint Hamilton cycles of $G$. 
\end{proof} 

\begin{proof}[Proof of Theorem~\ref{thm:main2}]
The proof is similar to the proof of Theorem~\ref{thm:main} and so we merely sketch it and point out the differences. 

Let $c>1/4$, $\eps>0$ where we may assume for the proof that $\eps$ is sufficiently small. Let $n$ be a sufficiently large integer. Let $d= cn$ and assume that $G$ is a balanced $d$-regular bipartite oriented graph on $2n$ vertices. Obviously, an oriented graph is a digraph, and so Lemmas~\ref{lem:partition} and~\ref{lem:pathcover} apply to this case just as above. 
Thus we obtain $K^3=\log^3n$ oriented subgraphs $H_1,\ldots, H_{K^3}$ satisfying the properties \ref{PPartition}--\ref{PsemiDegInWis} as in the previous proof.  
Furthermore, for every $i\in[K^3]$, we obtain a collection ${\bf P}^{(i)}$ of at least $r'=r-n^{24/25}\log n$ edge-disjoint path covers of $H_i[U_i]$, each of size at most $a=n/\log ^4 n$, and such that 
\eqref{eq:aux081} holds. 

Now fix $i\in [K^3]$ and let $\cP_1^{(i)},\ldots,\cP_{r'}^{(i)}$ be $r'$ of those path covers of ${\bf P}^{(i)}.$ We iteratively find $r'$ edge-disjoint cycles $C_1^{(i)},\ldots,C_{r'}^{(i)}$ in $H_i$ such that $C_k^{(i)}[U_i]$ consists exactly of the edges in $\cP_k^{(i)}$, for all $1\le k \le r'$. That is, again, the paths in $\cP_k^{(i)}$ are connected to a cycle $C_k^{(i)}$ via edges in $E(U_i,W_i)\cup E(W_i,U_i)\cup E(W_i)$.  
For $1\le k\le r'$ suppose that we have obtained such $k-1$ edge disjoint cycles $C_1^{(i)},\ldots,C_{k-1}^{(i)}$. Let $F_k$ be the graph obtained from $H_i$ by removing the edges of those $k-1$ cycles.
The argument why we can greedily pick pairwise distinct vertices $s_1,t_1,\ldots,s_{\ell},t_{\ell}\in W_i$ satisfying~\eqref{eq:connectingEdges} only differs in the constant factor in the lower bound 
$d_{H_i}^\pm (v,W_i)\ge \eps |W_i|/32K$, but the rest of the argument is essentially the same. 

Similarly, we obtain analogously to above that $\delta^{0}(F_k[W_i])\ge c'n'$ for some $c'>1/4$. 

Now instead of Lemma~\ref{lem:pathstohamcycle} we use Lemma~\ref{lem:pathstoLONGcycle} to find a collection $\{P_1,\ldots,P_{\ell}\}$ of pairwise vertex disjoint paths in $F_k[W_i]$ such that $P_j$ is an $s_j$-$t_j$-path for $1\le j\le \ell$. These paths, together with the paths in $\cP_k^{(i)}$ and the edges in~\eqref{eq:connectingEdges} form a cycle $C_{k}^{(i)}$ in $F_k\se H_i$ that is edge-disjoint from $C_1^{(i)},\ldots,C_{k-1}^{(i)}$ and from the paths in $\cP_{k+1}^{(i)},\ldots,\cP_{r'}^{(i)}$. Since $C_{k}^{(i)}$ covers all the vertices of $U_i$ this implies that the length of $C_{k}^{(i)}$ is at least $|U_i| = n-O(n/\log^2n).$ The rest is analogous to the proof above. 
\end{proof} 

It remains to prove Lemmas~\ref{lem:partition}, \ref{lem:pathcover}, \ref{lem:pathstohamcycle}, and~\ref{lem:pathstoLONGcycle}. As noted earlier, we move the proof of Lemma~\ref{lem:partition} to the appendix due to its similarity with its counterpart in~\cite{ferber2018}.

\begin{proof}[Proof of Lemma~\ref{lem:pathcover}]
Let $(A,B)$ be a bipartition of $H$ such that $|A|=|B|=m$, and let $b=2\log^4m$. Let $V_1^A,\ldots,V_b^A$ and $V_1^B,\ldots,V_b^B$ be partitions of $A$ and $B$ respectively, chosen independently and uniformly at random among all partitions such that $|V_i^A|=|V_i^B|= m/b$ for all $i$.  
For a fixed $i\in[b]$ and a fixed vertex $v\in A$, the random variable $d^+(v,V_i^B)$ has a hypergeometric distribution with parameters $(m,d^+(v),m/b)$. Therefore, the probability that $|d^+(v,V_i^B) - r/b|> (r/b)^{3/5}$ is at most $\exp(-(r/b)^{1/5}/6)$, by Remark~\ref{rem:hypergeom} and since $d^+(v,B)=r\pm r^{3/5}$ by assumption. A similar concentration argument applies to $d^-(v,V_i^B)$ as well as to $d^{\pm}(w,V_j^A)$ for every vertex $w\in B$ and $j\in [b]$. It follows by the union bound that with probability at least  $1-8mb\exp(- (r/b)^{1/5}/6)=1-o(1)$ we have that 
\begin{align}
d^\pm(v,V_i^B)=\frac{r}{b}\pm \left(\frac{r}{b}\right)^{3/5}\text{ for all }v\in A,\; i\in [b], \text{ and}\label{aux443}\\
d^\pm(w,V_j^A)=\frac{r}{b}\pm \left(\frac{r}{b}\right)^{3/5}\text{ for all }w\in B,\; j\in [b].\label{aux444}
\end{align}
Fix partitions of $A$ and $B$ that satisfy~\eqref{aux443} and~\eqref{aux444}.

Let $(W^A,W^B)$ denote a bipartition of the complete bipartite digraph $D_{b,b}$, where the elements of the two sets are labelled 
$W^A=\{w_j^A\mid 1\le j\le b\}$ and $W^B=\{w_j^B\mid 1\le j\le b\}$. 
Then $D_{b,b}$ contains $b$ edge-disjoint Hamilton paths, say $P_1,\ldots,P_b$, all of which have their start vertex in $W^A$, and such that no vertex in $W^A\cup W^B$ is the endpoint of more than $2\sqrt{\log b}$ of these paths, by Lemma \ref{lem:completebipartite}.

Let $P_1=w_{i_1}^A...w_{i_{2b}}^B$ and let $F_{1},\ldots,F_{2b-1}$ be the corresponding bipartite subgraphs of $H$ having edge sets 
$$E(V_{i_1}^A,V_{i_2}^B), E(V_{i_2}^B,V_{i_3}^A),\ldots, E(V_{i_{2b-1}}^A,V_{i_{2b}}^B),$$ 
respectively (recall that $E(V,W)$ denotes the set of all edges of a digraph that are oriented {\em from} $V$ {\em to} $W$).

For each $j\in [2b-1]$, we apply Lemma~\ref{lem:lem23} to the digraph $F_j$ (and keep Remark~\ref{rem:NearlyBalanced} in mind in case $|V_{i_j}^A|$ and $|V_{i_{j+1}}^B|$, say, differ by 1). Note that the assumptions are satisfied with slack for $m'=m/b$ and $r'=r/b-(r/b)^{3/5}$, by~\eqref{aux443} and~\eqref{aux444}. 
We conclude that $F_j$ contains at least 
\[\frac{r}{b}-\left(\frac{r}{b}\right)^{3/5}-\left(\frac{m}{b}\right)^{24/25}\geq \frac{r}{b}-2\left(\frac{m}{b}\right)^{24/25}\]
edge-disjoint matchings, each of size at least $(m/b)-(m/b)^{7/8}$. Moreover, every vertex in $V_{i_j}^A\cup V_{i_{j+1}}^B$ (or $V_{i_j}^B\cup V_{i_{j+1}}^A$, respectively) is contained in at least 
\[\frac{r}{b}-\left(\frac{r}{b}\right)^{3/5}-\left(\frac{m}{b}\right)^{24/25}-2\left(\frac{m}{b}\right)^{5/6}\geq \frac{r}{b}-2\left(\frac{m}{b}\right)^{24/25}\]
 of these matchings.

Note that, for each $j\in[2b-1]$, all edges of $F_j$ are oriented from $V_{i_j}^A$ to $V_{i_{j+1}}^B$ if $j$ is odd, and from $V_{i_j}^B$ to $V_{i_{j+1}}^A$ if $j$ is even. 
Therefore, we may pick an arbitrary such matching from $F_j$ for every $j\in [2b-1]$ and concatenate those matchings to form a path cover $\cP$ of $H$. 

Then $\cP$ contains at least $(2b-1)(m/b-(m/b)^{7/8})$ edges 
and so it must be of size at most $m/b + (2b-1)(m/b)^{7/8}\leq m/\log^4 m$, since each of the $2m$ vertices of $H$ is in exactly one of the paths of $\cP$. 

Iteratively picking distinct matchings for each $F_j$, we obtain $r/b-2(m/b)^{24/25}$ such path covers for $P_1$. We do the same for all $b$ Hamilton paths $P_1,\ldots,P_b$ of $D_{b,b}.$ Denote the union of all path covers obtained this way by ${\bf P}$, and note that ${\bf P}$ contains at least $b\left(r/b-2(m/b)^{24/25}\right)\geq r-m^{24/25}\log m$ path covers since $m$ is large enough. Since the paths $P_1,\ldots,P_b$ are pairwise edge-disjoint it follows that the path covers in ${\bf P}$ are pairwise edge-disjoint. 

It remains to show that the graph $G_{\bf P}$ has minimum semidegree at least $r-m/(\log m)^{39/10}$.
As noted above, for every bipartite graph $F_j$ of $P_1$, $1\le j <2b-1$, every vertex in $V_{i_j}^{A/B}$ is in at least $r/b-2(m/b)^{24/25}$ matchings. That is, every such $v$ has $d^+(v,V_{i_{j+1}}^{B/A})$ at least $r/b-2(m/b)^{24/25}$ in the graph formed by the union of those matchings. The same lower bound holds for every path $P_j$ and every $v$ that is not in the vertex class of the endpoint of $P_j$. Since a particular $V_{i_j}^{A/B}$ is the ``endpoint'' of at most $2\sqrt{\log b}$ of the paths $P_1,\ldots,P_b$ we get that for all $v\in V(H)$
\begin{align*}
    d^+(v)&\ge (b-2\sqrt{\log b}) \cdot  \left(\frac{r}{b}-2\left(\frac{m}{b}\right)^{24/25}\right)
    \geq r-\frac{m}{(\log m)^{39/10}}
\end{align*}
in the graph formed by the union $\bigcup \cP_i$ of all path covers. A similar argument applies to $d^-(v)$ in $G_{\bf P}$, which finishes the proof the lemma.
\end{proof}

\begin{proof}[Proof of Lemma~\ref{lem:pathstohamcycle}]
Let $(A,B)$ be a bipartition of $F$ such that $|A|=|B|=n'$. Choose a partition $W_1\dot\cup\ldots\dot\cup W_a$ of $A\cup B$ uniformly at random from all partitions that satisfy 
\begin{enumerate}[(a)]
\item $s_i,t_i\in W_i$ for all $i$,
\item $\left| |W_i|-|W_j|\right| \le 2$ for all $i,j$,
\item $|W_i^A|-|W_i^B|=|\{s_i,t_i\}\cap A|-1$  \label{prop:lem3.3}.
\end{enumerate} 

To see that such a partition exists let $S=\{s_1,t_1,\ldots, s_a,t_a\}$, let $I_A\se [a]$ be the set of indices such that $s_i,t_i\in A$, let $I_B\se [a]$ be the set of indices such that $s_i,t_i\in B$, and let $I_m=[a]\sm(I_A\cup I_B).$ Since $S$ is balanced, $|I_A|=|I_B|$ which we denote by $a'$. Let $A'= A\sm S$, $B'= B\sm S$ and assume first that $x=(n'-a-a')/a$ is an integer. Let $W_1'\dot\cup\ldots\dot\cup W_a'$ be a partition of $A'\cup B'$ such that $|W_i'\cap A| = x$ if $i\in I_A\cup I_m$, $|W_i'\cap A| = x+1$ if $i\in I_B$, and similarly, 
$|W_i'\cap B| = x$ if $i\in I_B\cup I_m$, $|W_i'\cap B| = x+1$ if $i\in I_B$. Note that this is possible by choice of $x$ and since $|I_A|=|I_B|$. Then the partition $W_1\dot\cup\ldots\dot\cup W_a$ is a partition as desired if we let $W_i=W_i'\cup\{s_i,t_i\}$ for all $i\in A$. In this case the bound in $\mathrm{(b)}$ is even 1. When $x$ is not an integer then a similar construction works (some occurrences of $x$ replaced by $\lfloor x\rfloor$ and some by $\lceil x\rceil$), in which case the set sizes may differ by 2. 

Fix $v\in V(F)$ and $i\in [a]$. 
Note that $d^+(v,W_i\sm\{s_i,t_i\})$ has a hypergeometric distribution with parameters $(n', d^+(v,V(F)\sm S), m)$, where $m=n'/a\pm 1$ and $d^+(v,V(F)\backslash S)\ge d^+(v)-a$. 
Therefore, for all $\eps >0$ the probability that $d^+(v,W_i)< (c'-\eps)n'/a$ is at most 
$\exp(-\eps^2n'/12a)$, since $d^+(v)\ge c'n'$ and 
by Remark~\ref{rem:hypergeom}. A similar bound holds for $d^-(v,W_i)$. Taking the union bound 
we deduce that with probability $1-4n'a\exp(- \eps^2 n'/12a)=1-o(1)$ 
\begin{equation}\label{aux:653}
d^\pm(v,W_i)\geq (c'-\eps)\frac{n'}{a} > \frac{m'+3}{2} \text{ for all } v\in V(F),\; i\in[a],
\end{equation}
where $m'=\min\{|W_i^A|,|W_i^B|\}$, $\eps$ satisfies $0<\eps<c'-1/2$, and we use that $a\ll n'/\log n'$.  

Fix a partition that satisfies~\eqref{aux:653}. We claim that this is sufficient to find a Hamilton $s_i$-$t_i$-path in $F[W_i]$, for every $i\in [a]$. The following implies this already when $s_i\in A$, $t_i\in B$ (or vice versa), when, by \ref{prop:lem3.3}, we have $|W_i^A|=|W_i^B|$. 
\begin{claim}\label{claim:HamPath} Let $m'$ be a non-negative integer and let $G=(A,B)$ be a bipartite digraph  such that $|A|=|B|=m'$. Let $x\in A$, $y\in B$. If $\delta^0(G)\ge m'/2+1$ then $G$ contains a Hamilton path from $x$ to $y$.\end{claim}
\begin{proof}[Proof of claim]
Let $A'=A\sm\{x\}$ and $B'=B\sm\{y\}$, and let $G'$ be the (undirected) bipartite graph with vertex set $V'=A'\cup B'$ and edge set $E'=\{ab : (b,a)\in E(G)\}$. 

We claim that $G'$ contains a perfect matching. Note that $d_{G'}(a)\ge d^-_G(a)-1\ge (m'-1)/2$ for all $a\in A'$ and $d_{G'}(b)\ge d^+_G(b)-1\ge (m'-1)/2$ for all $b\in B'$. Let now $X\se A'$ be non-empty and assume that $|N_{G'}(X)|<|X|$. Since every vertex in $X$ has at least $(m'-1)/2$ neighbours in $G'$ it follows that $|X|>(m'-1)/2$. Moreover, the set $B'\sm N_{G'}(X)$ is non-empty, so for any vertex $v\in B'\sm N_{G'}(X)$ we have $N_{G'}(v)\subseteq A'\sm X$. This, however, implies that $d_{G'}(v) \le |A'\sm X|<(m'-1)/2$, a contradiction. Thus, $|N_{G'}(X)|\ge |X|$ for all $X\se A'$, which implies that $G'$ contains a perfect matching, by Hall's Theorem. 

Let $\{(v_1,w_1),...,(v_{m'-1},w_{m'-1})\}$ denote the corresponding matching of directed edges in $G$ such that $v_i\in B'$ and $w_i\in A'$ for all $1\le i \le m'-1$, and let $w_{m'} = x$ and $v_{m'}=y$. 
Consider now the following auxiliary digraph $H$ on vertex set $V(H)=\{z_1,...,z_{m'}\}$. For each pair $(i,j)$ let $(z_i,z_j)$ be an edge of $H$ if $(w_i,v_j)$ is an edge of $G$. Note that $H$ satisfies $\delta^0(H)\ge\delta^0(G)-1\ge m'/2$. Therefore, $H$ contains a Hamilton cycle, say with edge set $C$, by Theorem~\ref{thm:ghouila-houri}. Now, this Hamilton cycle corresponds to a Hamilton path from $x$ to $y$ in $G$ which can be obtained by replacing each edge $(z_i,z_j)$ in $C$ by the edges $(w_i,v_j)$ and $(v_j,w_j)$ (the latter only if $j\neq m'$) in $G$. 
\end{proof}
Clearly this implies that every $F[W_i]$ has a Hamilton $s_i$-$t_i$-path in the case when $s_i\in W_i^A$ and $t_i\in W_i^B$, or vice versa. Assume now that both $s_i$ and $t_i$ are on the same side of the bipartition, say without loss of generality in $W_i^A$. In that case $|W_i^A|=|W_i^B|+1$ by \ref{prop:lem3.3}. The balanced bipartite digraph $F[(W_i^A\cup W_i^B) \sm\{s_i\}]$ satisfies the assumptions of the claim and thus contains a Hamilton path from $u$ to $t_i$ for any out-neighbour $u$ of $s_i$. Adding the edge $(s_i,u)$ to that path yields a Hamilton path from $s_i$ to $t_i$ in $F[W_i]$, as required. 
\end{proof}

\begin{proof}[Proof of Lemma~\ref{lem:pathstoLONGcycle}]
Let $(A,B)$ be a bipartition of $F$ such that $|A|=|B|=n'$. Similarly to the proof of Lemma~\ref{lem:pathstohamcycle} we choose a partition $W_1\dot\cup\ldots\dot\cup W_a$ of $A\cup B$ uniformly at random from all partitions that satisfy 
\begin{enumerate}[(a)]
\item $s_i,t_i\in W_i$ for all $i$,
\item $\left| |W_i|-|W_j|\right| \le 1$ for all $i,j$,
\item $|W_i^A| = |W_i^B|.$ 
\end{enumerate} 
Analogously to~\eqref{aux:653} we deduce that with probability $1-o(1)$ 
\begin{equation}\label{aux:654}
d^\pm(v,W_i)\geq (c'-\eps)\frac{n'}{a} > \frac{m'}{4} \text{ for all } v\in V(F),\; i\in[a], 
\end{equation}
where $m'=|W_i^A|.$ 
Fix a partition such that~\eqref{aux:654} is satisfied. We now find an $s_i$-$t_i$-path in $F[W_i]$ using the following. 
\begin{claim}
Let $G$ be a balanced bipartite oriented graph on $2m'$ vertices. Assume that the minimum semidegree of $G$ satisfies $\delta^{0}(G)> m'/4$. Then $G$ is strongly connected.
\end{claim}
\begin{proof}
Let $v$ be an arbitrary vertex in $G$ and let $R^+(v)$ be the set of vertices $w$  such that there is a $v$-$w$-path in $G$. We first show that $|R^+(v)| > m'$. 

Suppose not. Let $G'=G[R^+(v)]$. Then $\delta^+(G')> m'/4$ as all out-neighbours of all $w\in R^+$ are elements of $R^+(v)$, by definition. 
Since $G$ is bipartite, so is $G'$. Let $A\cup B$ be some bipartition of $G'$. By the minimum degree assumption, the set $E(A,B)$ has size greater than $|A|m'/4$, and so there is a vertex $b$ in $B$ of in-degree greater than $|A|m'/4|B|$. As the in-  and out-neighbours of $b$ are distinct elements of $A$ (since $G'$ is an oriented graph) we obtain that 
$$ |A| > \frac{m'}{4}\left(\frac{|A|}{|B|}+1\right).$$
Counting the edges in $E(B,A)$ gives analogously that 
$$ |B| > \frac{m'}{4}\left(\frac{|B|}{|A|}+1\right).$$
Combining the two inequalities implies that 
$$|R^+(v)|= |A|+|B| > \frac{m'}{4}\left(\frac{|A|}{|B|}+\frac{|B|}{|A|}+2\right)\ge m',$$
where the last step follows from the AM-GM inequality. 

Analogously one can show that the set $R^-(v)$ of vertices $w$ such that there is a $w$-$v$-path in $G$ has size greater than $m'$. 
Since this is true for any $v\in V(G)$, it follows that for any two vertices $v$ and $v'$ of $G$, the sets $R^+(v)$ and $R^-(v')$ intersect, that is, there is a path from $v$ to $v'$.  
\end{proof}
This finishes the proof of the lemma since all graphs $F[W_i]$ are balanced bipartite oriented graphs and satisfy the degree condition~\eqref{aux:654}. 
\end{proof}

\section{Conclusion}\label{sec:outro}

In this paper we prove that for every $c>1/2$ every $cn$-regular bipartite digraph on $2n$ vertices admits an almost decomposition of its edge set into Hamilton cycles, as long as $n$ is large enough. We also prove that for every $c>1/4$ every $cn$-regular bipartite oriented graph on $2n$ vertices admits an almost decomposition of its edge set into nearly Hamilton cycles, as long as $n$ is large enough. 
This gives a first approximate version of Conjecture~\ref{conj:jackson}. The following two would each constitute a strengthening towards Conjecture~\ref{conj:jackson}. 
\begin{conjecture}\label{conj:jackson-approx}
Let $c>1/2$ and let $n$ be sufficiently large. Then every $cn$-regular bipartite digraph $G$ on $2n$ vertices has a Hamilton cycle decomposition.
\end{conjecture}
Note that this is a bipartite analogue of~\cite[Theorem 1.4]{KUHN201362}: a digraph on $n$ vertices with minimum semidegree $cn$ for $c>1/2$ has a Hamilton decomposition, provided that $n \geq n_0(c)$.
\begin{conjecture}\label{conj:jackson-approx2}
Let $\eps>0$, let $n$ be sufficiently large, and let $d > n/4$ be an integer. Then every $d$-regular bipartite oriented graph on $2n$ vertices contains at least $(1-\eps)dn$ edge-disjoint Hamilton cycles. 
\end{conjecture}
The condition $d > n/4$ would be best possible since the oriented graph may be disconnected otherwise. Furthermore, the assumption of being regular is necessary for such a statement. To see this consider, for example, a blow up of a $C_4$ with slightly uneven vertex classes. This oriented graph has minimum semidegree slightly below $n/2$, yet fails to be Hamiltonian.

A further direction for exploration may be multi-partite tournaments. Let a {\em regular $r$-partite tournament} be a regular orientation of the complete $r$-partite graph $K(n;r)$ with equal size vertex classes. In~\cite{KUHN201362}, K\"uhn and Osthus not only prove Kelly's conjecture, but more generally, that every sufficiently large regular digraph $G$ on $n$ vertices whose degree is linear in $n$ and which is a {\em robust outexpander} contains a Hamilton cycle decomposition. In~\cite[Section 1.6]{ko2014} they then argue that, for $r\ge 4$, every sufficiently large $r$-partite tournament is a robust outexpander, and thus, has a Hamilton cycle decomposition. 
The approach via robust outexpanders does not cover the bipartite nor the tripartite case. Yet it is conjectured in~\cite{ko2014}, additionally to Jackson's conjecture, that every regular tripartite tournament has a Hamilton cycle decomposition. 

A possible approximate version of the conjecture for tripartite tournaments could be the following. 
\begin{conjecture}\label{tripartite-approx}
Let $\eps >0$, $c>1$ and let $n$ be sufficiently large. 
Let $G$ be a $cn$-regular tripartite digraph with vertex classes each of size $n$. Then $G$ contains at least $(1-\eps)cn$ edge-disjoint Hamilton cycles. 
\end{conjecture}
Parts of our arguments do work for such an approximate version. The equivalent of Claim~\ref{claim:HamPath}, however, does not seem to easily transfer. In fact, assuming just a lower bound of roughly $n$ on the minimum semidegree of a balanced tripartite digraph on $3n$ vertices does not necessarily imply that the graph is Hamiltonian.

\bigskip 
\noindent
{\bf Acknowledgement.}
The authors would like to thank Asaf Ferber for helpful discussions during the early stages of this project. Furthermore, we would like to thank the referee for many helpful comments that improved the paper. In particular we are grateful for suggesting to add a version of Theorem 1.4. 
 
\bibliographystyle{abbrv}
\bibliography{ref}

\appendix
\section{Proof of \autoref{lem:partition}}\hypertarget{appendix}{}
Select at random $K$ equipartitions of $A$ and $K$ equipartitions of $B$, each into $K^2$ sets: for each $i\in [K]$ let $\{S_{i,k}^A\}_{k=1}^{K^2}$ be the $i^{\mathrm{th}}$ partition of $A$ and let $\{S_{i,k}^B\}_{k=1}^{K^2}$ be the $i^{\mathrm{th}}$ partition of $B$. Note that all parts of all partitions have size either $\lfloor n/K^2\rfloor$ or $\lceil n/K^2\rceil$, and for each index $i$ and each vertex $v\in A$ (respectively $B$) there exists a unique index $k(i,v)$ such that $v\in S_{i,k(i,v)}^A$ (respectively $S_{i,k(i,v)}^B$). Denote by $S_{i,k}$ the union of $S_{i,k}^A$ and $S_{i,k}^B$.

Consider the following random sets. 
For $v\in V(D)$, $i\in  [K]$, let $X^\pm (v,i)$ 
be the set of vertices $u\in N_D^\pm(v)\cap S_{i,k(i,v)}$ such that $u,v\in S_{j,\ell}$ for some $j\neq i$ and some $\ell$. 
Further, let $Y^\pm(v)$ be the set of vertices $w\in N^\pm _D(v)$ such that both $v$ and $w$ are in the same set $S_{i,k}$ for some $i,k$.  In other words, if we colour the edges of all induced subgraphs $D[S_{i,k}]$ in colour $i$ (allowing multiple colours), $X^\pm(v,i)$ is the set of all vertices $w$ such that the edge $(v,w)$ (or $(w,v)$, respectively) received colour $i$ and at least one other colour, and $Y^\pm(v)$ is the set of vertices $w$ such that the edge $(v,w)$ (or $(w,v)$, respectively) received at least one colour.  Set $s=n/K^2$ and  $b=\E(|Y^\pm(v)|)$ where we note that $b$ is independent of $v$ since all degrees in $D$ are equal and since the partitions were chosen uniformly. 
We claim that all of the following properties hold with high probability: 
\begin{enumerate}[label={(\alph*)}]

\item For all $v\in V(D)$ and all sets $S_{i,k}$: 
	$d^\pm_D (v,S_{i,k}) =\frac{d |S_{i,k}|}{2n}\pm 2\sqrt{s\log n}$;
	\label{eq:lem1}

\item for all $v\in V(D)$ and $i\in [K]$:
	$|X^\pm(v,i)| =o(s);$ \label{eq:lem4}
	
\item for all $v\in V(D)$, $|Y^\pm(v)|=b\pm2\sqrt{K^2s\log n}.$
	\label{eq:lem5}
\end{enumerate}

For Property~\ref{eq:lem1} note that for fixed $v\in V(D)$, $i\in[K]$, and $k\in [K^2]$, both 
$d^+_D (v,S_{i,k})$ and $d^-_D (v,S_{i,k})$ are hypergeometric random variables, each with parameters $(n,d,|S_{i,k}|/2)$. Hence, it follows that~\ref{eq:lem1} holds with probability at least 
$1-16nK^3e^{-4\log n/3}=1-o(1)$, by Remark~\ref{rem:hypergeom} and the union bound. 

For fixed $v\in V(D)$ and $i\in [K]$, the random variable $|X^\pm(v,i)|$ is dominated by a binomial random variable with parameters $(nK,(1/K^2)^2)$. Thus $\E(|X^\pm(v,i)|)\leq \left(\frac{1}{K^2}\right)^2nK=o(s)$ and \ref{eq:lem4} follows from a straightforward application of Chernoff's inequality (Lemma~\ref{lem:chernoff}). 

For Property~\ref{eq:lem5} fix a vertex $v\in A$ and note that 
$$|Y^\pm(v)| = \left| N_D^\pm (v) \cap \bigcup_{i=1}^{K} S_{i,k(i,v)}^B\right|. 
	$$
For every $i\in [K]$ and every $w\in B$, the probability that $w\in S_{i,k(i,v)}^B$ is $1/K^2$. Thus, the probability that such a vertex $w$ is in $\bigcup_{i=1}^{K} S_{i,k(i,v)}^B$ is 
$p'= 1-(1-1/K^2)^K$. It follows that $b=\E(|Y^\pm(v)|)=dp'$. 
For each $i\in [K]$, let $U_i$ be a random subset of $B$, where every $w\in B$ is an element of $U_i$ with probability $1/K^2$, all choices being independent. Let $U =   \bigcup_{i=1}^{K} U_i$ and let $\mathcal{E}$ be the event that $|U_i| = |S_{i,k(i,v)}^B|$ for all $i$. Then the random variable $|N_D^\pm (v) \cap U|$ is binomially distributed with parameters $(d,p')$, and thus, $\E(|N_D^\pm (v) \cap U|) =b.$ Furthermore, the random variable $|Y^\pm(v)| $ has the same distribution as $|N_D^\pm (v) \cap U|$ conditioned on $\mathcal{E}$. Hence,  
\begin{equation}\label{eq:aux457}
\Pr\left(\left| |Y^\pm (v)| - b\right|>t \right) \le \Pr\left( \left| |N_D^\pm (v) \cap U| - b \right| >t\right) /\Pr(\mathcal{E})
\end{equation}
for all $t$. 
Now, each $|U_i|$ has a binomial distribution with mean $s$, thus $\Pr(|U_i|=j)$ is maximised when $j=s$. Thus, by independence, 
$$ \Pr(\mathcal{E}) = \prod_{i=1}^K \Pr(|U_i|=s) \ge 1/(n+1)^{-K}.$$
Hence, we deduce from~\eqref{eq:aux457} that 
$$\Pr\left(\left| |Y^\pm (v)| - b\right|>t \right)  
\le 2e^{- t^2/3b} (n+1)^K,$$
by Chernoff's inequality (Lemma~\ref{lem:chernoff}). If $t=2\sqrt{n\log n}$ then the expression on the right hand side is of order $o(1/n)$, where we use that $b= dp'\sim cn/K$. 
The same inequality holds for all vertices $v\in B$, so~\ref{eq:lem5} follows by taking the union bound over all $v\in V(D)$. 

Now fix $K$ partitions  $\{S_{i,k}^A\}_{k=1}^{K^2}$ of $A$, and $K$ partitions $\{S_{i,k}^B\}_{k=1}^{K^2}$ of $B$, such that~\ref{eq:lem1},~\ref{eq:lem4} and~\ref{eq:lem5} are satisfied. 

Let $D'$ be the digraph consisting of all edges of $D$ which are not contained in any $D[S_{i,k}]$. It follows directly from ~\ref{eq:lem5} that 
\begin{equation}\label{eq:aux284}
d^\pm_{D'}(v)=d-b\pm2\sqrt{K^2 s\log n} 
\end{equation}
for every $v\in V(D)$.

Relabel the sets $\{S_{i,k}\}_{(i,k) \in [K] \times [K^2]}$ as $W_1,\ldots,W_{K^3}$ and define the digraphs $H_j$ on vertex sets $W_j$ to be the edges of $D[W_j]$ that are not in $D[W_{j'}]$ for any $j'\neq j$. Finally, let $U_i = V(D)\sm W_i$. 

Property $(P1)$ of the lemma statement is trivially satisfied by definition. 
Furthermore, for every $1\le i \le K^3$ and every $v\in W_i$ we have that 
$$d^{\pm}_{H_i} (v,W_i) = \frac{d |W_{i}|}{2n}\pm \left(2\sqrt{s\log n} +o(s)\right),$$
by \ref{eq:lem1} and \ref{eq:lem4}. Hence, Property $(P4)$ follows 
since $d=cn$ and $|W_{i}|= n/K^2.$ 

It remains to choose edge sets $E_{H_i}(U_i,W_i)$, $E_{H_i}(W_i,U_i)$ and $E_{H_i}(U_i)$ such that properties $(P2)$ and $(P3)$ are satisfied. 
For a vertex $u\in V(D)$, let $I_u$ denote the set of indices $i$ such that $u\in W_i$, and note that by construction $|I_u|=K$. Furthermore, for an edge $e=(u,v)\in D'$ we have $I_u\cap I_v =\emptyset$ by definition of $D'$. 
Define random edge sets $E_1,\ldots,E_{K^3}$ and $D_1,\ldots,D_{K^3}$ as follows. 
For every edge $e=(u,v)\in D'$, add $e$ to exactly one of $E_1,\ldots,E_{K^3}, D_1,\ldots,D_{K^3}$ with the following probabilities. For each $i\in [K^3]$ 
\begin{itemize}
\item add $e$ to $E_i$ with probability $\frac{\eps}{2K}$ if $i\in I_u\cup I_v$; 
\item add $e$ to $D_i$ with probability $\frac{1-\eps}{K^3-2K}$ if $i\not\in I_u\cup I_v$,  
\end{itemize}
choices being independent for distinct edges. Note that the probabilities indeed add up to 1. 
Now for all $i\in [K^3]$ and all $v\in U_i$, 
$$\E(d^\pm _{D_i}(v))=d^\pm_{D'}(v)\frac{1-\eps}{K^3-2K}$$
and 
$$\E(d^\pm_{E_j}(v,W_j))=d^\pm_D(v,W_j)\frac{\eps}{2K}.$$
Hence by \eqref{eq:aux284}, Chernoff's inequality (Lemma~\ref{lem:chernoff}) and the union bound,  with probability at least $1-8nK^3 e^{-\omega(\log n)}=1-o(1)$ we have that $d_{D_i}^\pm(v)=r\pm r^{3/5}$ for all $i\in[K^3]$ and all $v\in U_i$, for some suitable $r= (1\pm \eps)d/K^3$. 
Similarly we obtain that with probability at least $1-4nK^3e^{-\omega(\log n)}=1-o(1)$, we 
have that for all $i\in[K^3]$ and all $v\in U_i$, 
$$ d_{E_i}^{\pm}(v,W_i) \ge \frac{\eps}{2K} \left(\frac{d |W_i|}{2n} - 2\sqrt{n/\log n}\right)
\ge \eps c |W_i|/8 K,$$
by \ref{eq:lem1}, Chernoff's inequality~\ref{lem:chernoff}, the union bound, and where we use in the last inequality that $d= cn$ and $|W_i|\gg \sqrt{n \log n}$. 

Finally, fix choices of $E_i$ and $D_i$ that satisfy 
$d_{D_i}^\pm(v)=r\pm r^{3/5}$ 
and  $d_{E_i}^{\pm}(v,W_i) \ge \eps c |W_i|/8 K$ for all $i\in[K^3]$ and all $v\in U_i$, and set $H_i = E_i \cup D_i \cup H_i[W_i]$. 

\end{document}